\documentclass[10pt]{amsart}
\usepackage{SC}
\fullpage

\begin{document}

\title{Scissors Congruence as K-theory}
\author{Inna Zakharevich}
\date{}
\maketitle

\begin{abstract}
Scissors congruence groups have traditionally been expressed algebraically in terms of group homology.  We give an alternate construction of these groups by producing them as the $0$-level in the algebraic $K$-theory of a Waldhausen category.
\end{abstract}

\section{Introduction}

The classical question of scissors congruence concerns the subdivisions of polyhedra.  Given a polyhedron, when is it possible to dissect it into smaller polyhedra and rearrange the pieces into a rectangular prism?  Or, more generally, given two polyhedra when is it possible to dissect one and rearrange it into the other?

In more modern language, we can express scissors congruence as a question about groups.  Let 
$\mathcal{P}(E^3)$ be the free abelian group generated by polyhedra $P$, quotiented out by the two relations $[P] = [Q]$ if $P\cong Q$, and $[P\cup P'] = [P] + [P']$ if $P\cap P'$ has measure $0$.  Can we compute $\mathcal{P}(E^3)$?  Can we construct a full set of invariants on it?  More generally, let $X$ be $E^n$ (Euclidean space), $S^n$, or $H^n$ (hyperbolic space).  We define a simplex of $X$ to be the convex hull of $n+1$ points in $X$ (restricted to an open hemisphere if $X =S^n$), and a polytope of $X$ to be a finite union of simplices.  Let $G$ be any subgroup of the group of isometries of $X$.  Then we can define a group $\mathcal{P}(X,G)$ to be the free abelian group generated by polytopes $P$ of $X$, modulo the two relations $[P] = [g\cdot P]$ for any polytope $P$ and $g\in G$, and $[P\cup P'] = [P] + [P']$ for any two polytopes $P,P'$ such that $P\cap P'$ has measure $0$.  The goal of Hilbert's third problem is to classify these groups.

The traditional approach to this question has been using the tools of group homology.  It is easy to see from the definition above that 
$$\mathcal{P}(X,G) = H_0(G, \mathcal{P}(X,\{\})),$$
and more generally that for any normal subgroup $H$ of $G$, 
$$\mathcal{P}(X,G) = H_0(G/H, \mathcal{P}(X,H)).$$
To see a detailed exploration of this perspective, see \cite{dupont01}.

On the other hand, the question of scissors congruence is highly reminiscent of a K-theoretic question.  Algebraic K-theory classifies projective modules according to their decompositions into smaller modules; topological K-theory classifies vector bundles according to decompositions into smaller-dimensional bundles.  Thus it is reasonable to ask whether scissors congruence can also be expressed as a K-theoretic question about polytopes being decomposed into smaller polytopes.

This question is made somewhat more difficult because while $\mathcal{P}(X,G)$ is exactly the $G$-coinvariants of the group completion of the monoid of polytopes under union, we cannot express it in the other order.  Since we only defined union when our polytopes had measure-$0$ intersection we cannot directly give the scissors congruence group as a group completion.  Instead we construct this question more indirectly.

In his book \cite{sah79}, Sah defined a notion of ``abstract scissors congruence'', an abstract set of axioms that are sufficient to define a scissors congruence group.  Inspired by this, we define a ``polytope complex'' to be a category which contains enough information to encode scissors congruence information. We then use a modified Q construction to construct a category which encodes the movement of polytopes and which has a K-theory spectrum associated to it.  We then prove that $K_0$ of this category is exactly $\mathcal{P}(X,G)$.

Our main goal is to prove a slight generalization the following theorem.  (For the full statement, see theorem \ref{thm:Wald}.)
\begin{theorem}
Given a manifold $X$, equal to $S^n$, $E^n$, or $H^n$, and a subgroup $G$ of the isometries of $X$, there exists a Waldhausen category $\SC(X)$ such that $K_0\SC(X) \cong \mathcal{P}(X,G)$.
\end{theorem}

The organization of this paper is as follows.  Section 2 introduces the two categorical constructions we will use: a Grothendieck construction indexed over finite sets, and a double category.  Section 3 explores the abstract notion of a polytope, and introduces the types of objects for which we can define scissors congruence categories.  Sections 4 and 5 define the Waldhausen category structure and give some examples of the kinds of spectra we can produce.  Section 6 consists of technical results and the proof of the main theorem.

\section{Categorical Preliminaries}

\subsection{Grothendieck Twists}

\begin{definition} 
  Given a category $\D$ we define the contravariant functor
  $F_\D:\FinSet^\op\rightarrow \mathbf{Cat}$ by $I\mapsto
  \D^I$.  The \textsl{Grothendieck twist of $\D$}, written $\Tw(\D)$,
  is defined to be the following Grothendieck construction applied to
  $F_\D$.  We let the objects of $\Tw(\D)$ be pairs $I\in
  \FinSet$, and $x\in \D^I$.  A morphism $(I,x)\rightarrow
  (J,y)$ in $\Tw(\D)$ will be a morphism $I\rightarrow J\in
  \FinSet$, together with a morphism $y\rightarrow
  F_\D(f)(x)\in \D^I$.  We will often refer to the function $I\rightarrow J$ as the \textsl{set
    map} of a morphism.
\end{definition}

$\Tw(\D)$ is the Grothendieck construction
$(\int_{\FinSet^\op}F_\D^\op)^\op$.  More concretely, an
object of $\Tw(\D)$ is a finite set $I$ and a set map $f:I\rightarrow
\ob\D$; we will write an object of this form as $\SCob{a}{i}$, with
the understanding that $a_i = f(i)$.  A morphism
$\SCob{a}{i}\rightarrow \SCob{b}{j}\in\Tw(\D)$ is a pair consisting of
a morphism of finite sets $f:I\rightarrow J$, together with morphisms
$F_i:a_i\rightarrow b_{f(i)}$ for all $i\in I$.



In general we will denote a morphism of $\Tw(\D)$ by a lower-case letter.  By an abuse of notation, we will use the same letter to refer to the morphism's set map, and the upper-case of that letter to refer to the $\D$-components of the morphism (as we did above).  If a morphism $f:\SCob{a}{i}\rightarrow \SCob{b}{j}$ has its set map equal to the identity on $I$ we will say that $f$ is a \textsl{pure $\D$-map}; if instead we have $F_i:a_i\rightarrow b_{f(i)}$ equal to the identity on $a_i$ for every $i$ we will say that $f$ is a \textsl{pure set map}.

If we consider an object $\SCob{a}{i}$ of $\Tw(\D)$ to be a formal sum $\sum_{i\in I}a_i$ then we see that $\Tw(\D)$ is the category of formal sums of objects in $\D$.  Then we have a monoid structure on the isomorphism classes of objects of $\Tw(\D)$ (with addition induced by the coproduct).  In later sections we will investigate the group completion of this monoid, but for now we will examine the structures which are preserved by this construction.

Much of $\Tw(\D)$'s structure comes from ``stacking'' diagrams of $\D$, so it stands to reason that much of $\D$'s structure would be preserved by this construction.  The interesting consequence of this ``layering'' effect is that even though we have added in formal coproducts, computations with these coproducts can often be reduced to morphisms to singletons.  Given any morphism $f:\SCob{a}{i}\rightarrow \SCob{b}{j}$ we can write it as
$$
\coprod_{j\in J} \left(\begin{diagram}\{a_i\}_{i\in f^{-1}(j)} & \rTo^{\left.f\right|_{f^{-1}(j)}} & \{b_j\}\end{diagram}\right)
$$

\begin{lemma} \lbl{lem:pullbacks} If $\D$ has all pullbacks then
  $\Tw(\D)$ has all pullbacks.  The pullback of the diagram
  \begin{diagram}
    \SCob{a}{i} & \rTo^f & \SCob{c}{k} & \lTo^g & \SCob{b}{j}
  \end{diagram} 
  is $\{a_i\times_{c_k} b_j\}_{(i,j)\in I\times_KJ}$.
\end{lemma}


However, sometimes the categories we will be considering will not be
closed under pullbacks.  It turns out, however, that if we are simply
removing some objects which are ``sources'' then $\Tw(\D)$ will still
be clsoed under pullbacks.

\begin{lemma} \lbl{lem:remsource} Let $\C$ be a full subcategory of
  $\D$ which is equal to its essential image, and let $\D'$ be the
  full subcategory of $\D$ consisting of all objects not in $\C$.
  Suppose that $\C$ has the property that for any $A\in \C$, if
  $\Hom(B,A)\neq \varnothing$ then $B\in \C$.   Then if $\D$ has all
  pullbacks, so does $\Tw(\D')$.
\end{lemma}

\begin{proof} Let $U:\Tw(\D') \rightarrow \Tw(\D)$ be the inclusion
  induced by the inclusion $\D'\rightarrow \D$.  We define a projection functor $P:\Tw(\D)\rightarrow
  \Tw(\D')$ by $P(\SCob{a}{i}) = \{a_i\}_{i\in I'}$, where $I'= \{i\in
  I\,|\, a_i\not\in \C\}$.

Suppose that we are given a diagram $A\rightarrow C \leftarrow B$ in
$\Tw(\D')$.  Let $X$ be the pullback of $UA \rightarrow UC
\leftarrow UB$ in $\Tw(\D)$; we claim that $PX$ is the pullback
of $A\rightarrow C \leftarrow B$ in $\Tw(\D')$.  Indeed,
suppose we have a cone over our diagram with vertex $D$, then $UD$
will factor through $X$, and thus $PUD = D$ will factor through $PX$.
Checking that this factorization is unique is trivial.
\end{proof}

We finish up this section with a quick result about pushouts.  It's clear that $\Tw(\D)$ has all finite coproducts, since we compute it by simply taking disjoint unions of indexing sets.  However, it turns out that a lot more is true.

\begin{lemma} \lbl{lem:twpushouts}
If $\D$ contains all finite connected colimits then $\Tw(\D)$ has all pushouts.
\end{lemma}

\begin{proof}
  Consider a morphism $f:\SCob{a}{i}\rightarrow \SCob{b}{j} \in
  \Tw(\D)$.  We can factor $f$ as a pure $\D$-map followed by a pure set map.
  Thus to show that $\Tw(\D)$ contains all pushouts it suffices to
  show that $\Tw(\D)$ contains all pushouts along pure set maps and pure $\D$-maps separately.  

  Now suppose that we are given a diagram
  \begin{diagram}
    \SCob{c}{k} & \lTo^g & \SCob{a}{i} & \rTo^f & \SCob{b}{j}
  \end{diagram}
  It suffices to show that the pushout exists whenever $g$ is a pure set-map or
  a pure $\D$-map.  Suppose that $g$ is a pure set map.  For $x\in J\cup_IK$ we
  will write $I_x$ (resp. $J_x$, $K_x$) for those elements in $I$ (resp. $J$,
  $K$) which map to $x$ under the pushout morphisms.  The pushout of the above
  diagram in this case will be $\{d_x\}_{x\in J\cup_IK}$, where $d_x$ is defined
  to be the colimit of the following diagram (if it exists in $\D$).  The
  diagram will have objects $a_i,b_j,c_k$ for all $i\in I_x$, $j\in J_x$ and
  $k\in K_x$.  There will be an identity morphism $a_i\rightarrow c_{g(i)}$ and
  a morphism $F_i: a_i \rightarrow b_{f(i)}$ for all $i\in I_x$.  Note that this
  colimit must be connected, since otherwise $x$ wouldn't be a single element in
  $J \cup_I K$.

  If $g$ is a pure $\D$-map the pushout of this diagram will be $\{d_j\}_{j\in
    J}$, where $d_j$ is defined to be the colimit of the diagram
  \begin{diagram}[small]
    \coprod_{i\in I_j} c_i & \lTo^{\coprod_{i\in I_j}G_i} &
    \coprod_{i\in I_j} a_i & \rTo^{\coprod_{i\in I_j}F_i} & b_j
    && I_j = f^{-1}(j).
  \end{diagram}
  which exists as the diagram is connected.  (Also, while we wrote the above diagrams using coproducts, they do not
  actually need to exist in $\D$.  In that case, we just expand the
  coproduct in the diagram into its components to produce a diagram
  whose colimit exists in $\D$.)
\end{proof}

\begin{remark}
  In order for $\D$ to contain all finite connected colimits it
  suffices for it to contain all pushouts and all coequalizers.  If
  $\D$ has all pushouts (but not necessarily all coequalizers) then
  examination of the proof above shows that $\Tw(\D)$ must be closed
  under all pushouts along morphisms with injective set maps.
\end{remark}

\subsection{Double Categories}

We will be using the notion of double categories originally introduced by
Ehresmann in \cite{ehresmann63}; we follow the conventions used by Fiore, Paoli
and Plonk in \cite{fiorepaoliplonk08}.

\begin{definition}
A \textsl{small double category} $\C$ is a set of objects $\ob\C$ together with two sets of morphisms $\Hom_v(A,B)$ and $\Hom_h(A,B)$ for each pair of objects $A,B\in \ob\C$, which we will call the \textsl{vertical} and \textsl{horizontal} morphisms.  We will draw the vertical morphisms as dotted arrows, and the horizontal morphisms as solid arrows.  $\C$ with only the morphisms from the vertical (resp. horizontal) set forms a category which will be denoted $\C_v$ (resp. $\C_h$).  

In addition, a double category contains the data of ``commutative squares'', which are diagrams
\begin{diagram}[small]
A & \rTo^\sigma & B\\
\dSub^p && \dSub_q \\
C & \rTo^\tau & D
\end{diagram}
which indicate that ``$q\sigma = \tau p$''.  Commutative squares have to satisfy certain composition laws, which we omit here as they simply correspond to the intuition that they should behave just like commutative squares in any ordinary category.

Given two small double categories $\C$ and $\D$, a \textsl{double functor} $F:\C\rightarrow \D$ is a pair of functors $F_v:\C_v\rightarrow \D_v$ and $F_h:\C_h\rightarrow \D_h$ which takes commutative squares to commutative squares.  We will denote the category of small double categories by $\mathbf{DblCat}$.
\end{definition} 

\begin{remark}
A small double category is an internal category object in $\mathbf{Cat}$.  We do not use this definition here, however, since it obscures the inherent symmetry of a double category.
\end{remark}

In general we will label vertical morphisms with latin letters and horizontal morphisms with Greek letters.  We will also say that a diagram consisting of a mix of horizontal and vertical morphisms commutes if any purely vertical (resp. horizontal) component commutes, and if all components mixing the two types of maps consists of squares that commute in the double category structure.

Now suppose that $\C$ is a small double category.  We can define a double category $\Tw(\C)$ by letting $\ob\Tw(\C) = \ob \Tw(\C_h)$ (which are the same as $\ob\Tw(\C_v)$ so there is no breaking of symmetry).  We define the vertical morphisms to be the morphisms of $\Tw(\C_v)$ and the horizontal morphisms to be the morphisms of $\Tw(\C_h)$.  In addition, we will say that a square
\begin{diagram}[small]
\SCob{a}{i} & \rTo^\sigma & \SCob{b}{j} \\
\dSub^p && \dSub_q \\
\SCob{c}{k} &\rTo^\tau & \SCob{d}{l}
\end{diagram}
commutes if for every $i\in I$ the square
\begin{diagram}[small]
a_i & \rTo^{\Sigma_i} & b_{\sigma(i)} \\
\dSub^{P_i} && \dSub_{Q_{\sigma(i)}} \\
c_{p(i)} & \rTo^{T_{p(i)}} & d_{\tau(p(i))}
\end{diagram}
commutes.  It is easy to check that with this definition $\Tw(\C)$ forms a double category as well, and in fact that $\Tw$ is a functor $\mathbf{DblCat}\rightarrow \mathbf{DblCat}$.

\section{Abstract Scissors Congruence}

In this section we will deal with scissors congruence of abstract
objects.

\begin{definition} 
  A \textsl{polytope complex} is a double category $\C$ satisfying the
  following properties:
  \begin{itemize}
  \item[(V)] Vertically, $\C$ is a preorder which has a unique initial object
    and is closed under pullbacks.  In addition, $\C$ has a Grothendieck
    topology.
  \item[(H)] Horizontally, $\C$ is a groupoid.
  \item[(P)] Every horizontal morphism $\Sigma:\bd A&\rTo& B\ed$ induces an
    equivalence of categories $\Sigma^*:\bd (\C_v\downarrow B) & \rTo & (\C_v
    \downarrow A)\ed$ with the property that
    for any vertical morphism $P:\bd B' &\rSub& B\ed$ there exists a unique
    horizontal morphism $\bd \Sigma^*B' & \rTo & B'\ed$ such that the square
    \begin{diagram}[small] \Sigma^*B' & \rTo & B'\\ \dSub^{\Sigma^*P} && \dSub_P \\
      A & \rTo & B
    \end{diagram} commutes.
  \item[(C)] If $\{\bd X_\alpha& \rSub & X\ed\}_{\alpha \in A}$ is a
    set of vertical morphisms which is a covering family of $X$, and
    $\Sigma:\bd Y & \rTo & X\ed$ is any horizontal morphism, then
    $\{\bd \Sigma^*X_\alpha& \rSub& Y\ed\}_{\alpha\in A}$ is a covering
    family of $Y$.
  \item[(B)] If $\{\bd X_\alpha & \rSub & X\ed\}_{\alpha\in A}$ is a set
    of maps such that for each $\alpha\in A$ there exists a covering
    family $\{\bd X_{\alpha\beta}& \rSub & X_\alpha \ed\}_{\beta\in
      B_\alpha}$ such that the refinement $\{\bd X_{\alpha\beta}& \rSub &
    X\ed\}_{\alpha\in A,\ \beta\in B_\alpha}$ is also a covering family,
    then $\{\bd X_\alpha& \rSub & X \ed\}_{\alpha\in A}$ was also
    originally a covering family.
  \end{itemize}

  A \textsl{polytope} is a noninitial object of $\C$.  The full double
  subcategory of polytopes of $\C$ will be denoted $\C_p$.  We will say that two
  polytopes $a,b\in \C$ are \textsl{disjoint} if there exists an object $c\in
  \C$ with vertical morphisms $\bd a & \rSub & c\ed$ and $\bd b & \rSub & c\ed$
  such that the pullback $a\times_c b$ is the vertically initial object.
\end{definition}


The main motivating example that we will refer to for intuition will be the example of Euclidean scissors congruence. Let the polytopes of $\C$ be polygons in the Euclidean plane, where we define a polygon to be a finite union of nondegenerate triangles.  We define the vertical morphisms of $\C$ to be set inclusions (where we formally add in the empty set to be the vertically initial object).  The topology on $\C_v$ will be the usual topology induced by unions; concretely, $\{\bd P_\alpha&\rSub&P\ed\}_{\alpha \in A}$ will be a cover if $\bigcup_{\alpha\in A} P_\alpha = P$.  For every $g\in E(2)$ (the group of Euclidean transformations of the plane) and polygon $P$ we define a horizontal morphism $g:\bd P&\rTo&g\cdot P\ed$.

Then axiom (V) simply says that the intersection of two polygons is either another polygon or else has measure $0$ (and therefore we define it to be the empty set).  Axiom (H) is simply the statement that $E(2)$ is a group.  Axiom (P) says that if we have polygons $P$ and $Q$ and a Euclidean transformation $g$ that takes $P$ to $Q$, then any polygon sitting inside $P$ is taken to a unique polygon inside $Q$.  Axiom (C) says that Euclidean transformations preserve unions.  Axiom (B) says that if you have a set of polygons $\{P_\alpha\}_{\alpha \in A}$ and sets $\{P_{\alpha\beta}\}_{\beta\in B_\alpha}$ such that 
$$\bigcup_{\beta\in B_\alpha}{P_{\alpha\beta}} = P_\alpha\qquad\hbox{and}\qquad\bigcup_{\alpha\in A}\bigcup_{\beta\in B_\alpha} P_{\alpha\beta} = P$$ then we must have originally had $P = \bigcup_{\alpha\in A} P_\alpha = P$.

In order to define scissors congruence groups we want to look at the formal sums
of polygons, and quotient out by the relations that $[P] = [Q]$ if $P\cong Q$,
and if we have a finite set of polygons $\{P_i\}_{i\in I}$ which intersect only
on the boundaries that cover $P$ then $[P] = \sum_{i\in I} [P_i]$.  Using a
Grothendieck twist we can construct a category whose objects are exactly formal
sums of polygons, and whose isomorphism classes will naturally quotient out the
first of these relations.  Thus we can now draw our attention to the second
relation, which concerns ways of including smaller polygons into larger ones.
In the language of polytope complexes, we want to understand the vertical
structure of $\Tw(\C)$.

We start with some results about how to move vertical information along
horizontal morphisms.  $\C$ has the property that ``pullbacks exist'', namely
that if we have the lower-right corner of a commutative square consisting of a
vertical and a horizontal morphism then we can complete it to a commutative
square in a suitably universal fashion.  It turns out that $\Tw(\C_p)$ has the
same property.

\begin{lemma} \lbl{TwCpullback}
Given any diagram 
\begin{diagram}
A & \rTo^\sigma & B & \lSub^q & B'
\end{diagram}
in $\Tw(\C_p)$, let $(\Tw(\C_p)\downarrow (A,B'))$ be the category whose objects
are commutative squares
\begin{diagram}[small]
A' & \rTo^\tau & B'\\
\dSub^p && \dSub_q\\
A & \rTo^\sigma & B
\end{diagram}
and whose morphisms are commutative diagrams
\begin{diagram}[small]
 &  & A'_1 &  &  \\
A &\ldSub(2,1)^{p_1}& \dSub_r &\rdTo(2,1)^{\tau_2}& B' \\
& \luSub(2,1)^{p_2} & A'_2 & \ruTo(2,1)^{\tau_2} & 
\end{diagram}
Then $(\Tw(\C_p)\downarrow (A,B'))$ has a terminal object.
\end{lemma}

We will refer to this terminal object as the \textsl{pullback} of the diagram,
and the square that it fits into a \textsl{pullback square}.  We denote this
terminal object by $\sigma^*B'$, and call the vertical morphism
$\sigma^*q:\bd\sigma^*B' & \rSub & A\ed$ and the horizontal morphism
$\tilde\sigma:\bd \sigma^*B' & \rTo & B'\ed$.  Note that if $\sigma$ (resp. $q$)
is an isomorphism, then so is $\tilde \sigma$ (resp. $\sigma^*q$).

\begin{proof}
Let $A' = \{\Sigma_i^*b'_{j'}\}_{(i,j')\in I\times_J J'}$.  We also define
$\tau:\bd A' &\rTo & B'\ed$ by the set map $\pi_2:I\times_JJ' \rightarrow J'$
and the horizontal morphisms $\Sigma_i:\bd\Sigma_i^*b'_{j'} & \rTo & b'_{j'}\ed$
and $p:\bd A' & \rSub & A\ed$ by the set map $\pi_1:I\times_JJ'\rightarrow I$
and the vertical morphisms $\Sigma_i^*(\bd b'_{j'} & \rSub & b_{q(j)}\ed)$.
Then these complete the original diagram to a commutative square by definition;
the fact that it is terminal is simple to check.
\end{proof}

Our second relation between polygons had the condition that we needed polygons
to be disjoint, so we restrict our attention to vertical morphisms which have
only disjoint polygons in each ``layer.''

\begin{definition} Given a vertical morphism $p:\bd\SCob{a}{i}&\rSub&
  \SCob{b}{j}\ed\in \Tw(\C_p)_v$ we say that $p$ is a \textsl{sub-map} if for
  every $j\in J$ and any two distinct $i,i'\in p^{-1}(j)$ we have
  $a_i\times_{b_j} a_{i'} = \initial$ in $\C_v$.  We will say that a sub-map $p$
  is a \textsl{covering sub-map} if for every $j\in J$ the sets $\{\bd a_i &
  \rSub & b_j\ed\}_{i\in p^{-1}(j)}$ are covers according to the topology on
  $\C_v$.

We will denote the subcategory of sub-maps by $\Tw(\C_p)_v^\Sub$.
\end{definition}

In the polygon example, a sub-map is simply the inclusion of several polygons which have measure-$0$ intersection into a larger polygon.  A covering sub-map is such an inclusion which is in fact a tiling of the larger polygon.  For example:

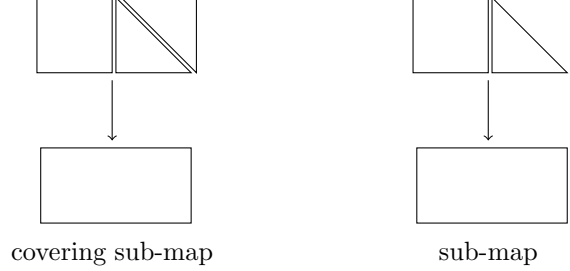
\begin{center}
\begin{tikzpicture}
  \draw (0,1) rectangle (1,2);
  \draw (1.05, 1) -- (2.05, 1) -- (1.05, 2) -- cycle;
  \draw (1.12, 2) -- (2.12, 2) -- (2.12, 1) -- cycle;

  \draw (0.05,-1) rectangle (2.05, 0);
  \draw [->] (1,.9) -- (1,.1);
  \node at (1,-1) [label=below:covering sub-map] {};

  \draw (5,1) rectangle (6,2);
  \draw (6.05, 1) -- (7.05, 1) -- (6.05, 2) -- cycle;

  \draw (5.05,-1) rectangle (7.05, 0);
  \draw [->] (6,.9) -- (6,.1);
  \node at (6,-1) [label=below:sub-map] {};

\end{tikzpicture}
\end{center}

From this point onwards in the text all vertical morphisms of $\Tw(\C_p)$ will
be sub-maps.  If a sub-map is in fact a covering sub-map we will denote it by
$\bd A & \rCover & B\ed$.  We will also refer to horizontal morphisms as
\textsl{shuffles} for simplicity.  From lemma \ref{lem:remsource} we know that
$\Tw(\C_p)_v$ has all pullbacks, and it is easy to see that the pullback of a
sub-map will also be a sub-map.  From axiom (B) we know that if $\{\bd X_\alpha
& \rSub & X\ed\}_{\alpha \in A}$ is a covering family with $X_{\alpha_0} =
\initial$ for some $\alpha_0\in A$, then the family $\{\bd X_\alpha & \rSub &
X\ed\}_{\alpha\in A\backslash \{\alpha_0\}}$ is also a covering family.  Thsu
the pullback of a covering sub-map will be another covering sub-map, which means
that not only is $\Tw(\C_p)_v^\Sub$ a category which has all pullbacks, but in
fact the Grothendieck topology on $\C_v$ induces a Grothendieck topology on
$\Tw(\C_p)_v$.  It turns out that the pullback functor defined above acts
continuously with respect to this topology.

\begin{lemma} \lbl{lem:pullandpush}
Let $\sigma:\bd A& \rTo& B\ed$ be a shuffle.

\begin{enumerate}
\item We have a functor $\sigma^*:(\Tw(\C_p)_v^\Sub\downarrow B) \rightarrow
  (\Tw(\C_p)_v^\Sub\downarrow A)$ given by pulling back along $\sigma$.
  This functor preserves covering sub-maps.
\item $\sigma^*$ has a right adjoint $\sigma_*$ which also preserves covering sub-maps.  If $\sigma$ has an injective set map  then $\sigma^*\sigma_* \cong 1$; if $\sigma$ has a surjective set map then $\sigma_*\sigma^* \cong 1$.
\end{enumerate}
\end{lemma}

Intuitively speaking, $\sigma^*$ looks at how each polytope in the image is
decomposed and decomposes its preimages accordingly.  $\sigma_*$ figures out
what the minimal subdivision of the image that pulls back to a refinement of the
domain is.  In our polygon example, we have the following:

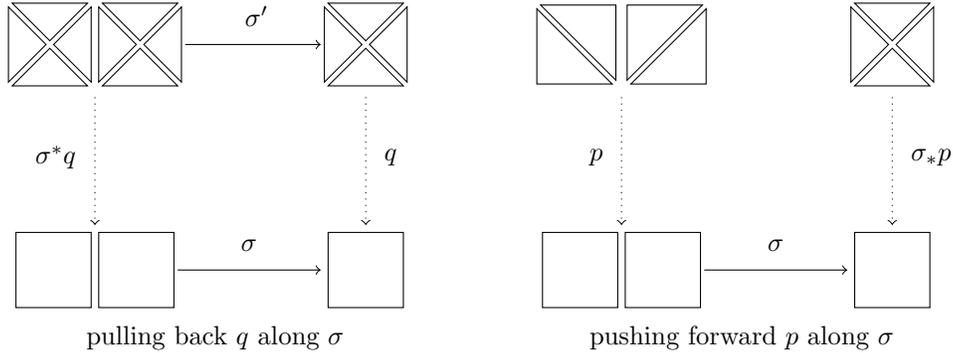
\begin{center}
\begin{tikzpicture}
  \path (0,0) coordinate (P);

  \draw (P) ++(-.05,0) rectangle +(1,1)
      ++(1.1,0) rectangle +(1,1);
  \draw (P) ++(4.1,0) rectangle +(1,1);

  \path (P) ++(4.1,3) coordinate (Q);
  \draw (Q) ++(0,-.05) -- +(1,0) -- +(.5,.5) -- cycle;
  \draw (Q) ++(-.05,0) -- +(0,1) -- +(.5,.5) -- cycle;
  \draw (Q) ++(1.05,0) -- +(0,1) -- +(-.5,.5) -- cycle;
  \draw (Q) ++(0,1.05) -- +(1,0) -- +(.5,-.5) -- cycle;

  \path (P) ++(-.1,3) coordinate (Q);
  \draw (Q) ++(0,-.05) -- +(1,0) -- +(.5,.5) -- cycle;
  \draw (Q) ++(-.05,0) -- +(0,1) -- +(.5,.5) -- cycle;
  \draw (Q) ++(1.05,0) -- +(0,1) -- +(-.5,.5) -- cycle;
  \draw (Q) ++(0,1.05) -- +(1,0) -- +(.5,-.5) -- cycle;
  \path (P) ++(1.1, 3) coordinate (Q);
  \draw (Q) ++(0,-.05) -- +(1,0) -- +(.5,.5) -- cycle;
  \draw (Q) ++(-.05,0) -- +(0,1) -- +(.5,.5) -- cycle;
  \draw (Q) ++(1.05,0) -- +(0,1) -- +(-.5,.5) -- cycle;
  \draw (Q) ++(0,1.05) -- +(1,0) -- +(.5,-.5) -- cycle;

  \draw (P) [->] ++(2.1,.5) -- +(1.9,0);
  \path (P) ++(3.05,.5) node [label=above:$\sigma$] {};
  \draw (P) [dotted,->] ++(4.6,2.8) -- +(0,-1.7);
  \path (P) ++(4.6,2) node [label=right:$q$] {};
  \draw (P) [->] ++(2.2,3.5) -- +(1.8,0);
  \path (P) ++(3.15,3.5) node [label=above:$\sigma'$] {};
  \draw (P) [dotted,->] ++(1,2.8) -- +(0,-1.7);
  \path (P) ++(1,2) node [label=left:$\sigma^*q$] {};
  \path (P) ++(2.6,0) node [label=below:pulling back $q$ along $\sigma$] {};

  \path (7,0) coordinate (P);

  \draw (P) ++(-.05,0) rectangle +(1,1)
      ++(1.1,0) rectangle +(1,1);
  \draw (P) ++(4.1,0) rectangle +(1,1);

  \path (P) ++(4.1,3) coordinate (Q);
  \draw (Q) ++(0,-.05) -- +(1,0) -- +(.5,.5) -- cycle;
  \draw (Q) ++(-.05,0) -- +(0,1) -- +(.5,.5) -- cycle;
  \draw (Q) ++(1.05,0) -- +(0,1) -- +(-.5,.5) -- cycle;
  \draw (Q) ++(0,1.05) -- +(1,0) -- +(.5,-.5) -- cycle;

  \path (P) ++(-.1,3) coordinate (Q);
  \draw (Q) ++(-.025,-.025) -- +(1,0) -- +(0,1) -- cycle;
  \draw (Q) ++(.025,1.025) -- +(1,0) -- +(1,-1) -- cycle;
  \path (P) ++(1.1, 3) coordinate (Q);
  \draw (Q) ++(-.025,1.025) -- +(1,0) -- +(0,-1) -- cycle;
  \draw (Q) ++(1.025,-.025) -- +(0,1) -- +(-1,0) -- cycle;

  \draw (P) [->] ++(2.1,.5) -- +(1.9,0);
  \path (P) ++(3.05,.5) node [label=above:$\sigma$] {};
  \draw (P) [dotted,->] ++(4.6,2.8) -- +(0,-1.7);
  \path (P) ++(4.6,2) node [label=right:$\sigma_*p$] {};
  \draw (P) [dotted,->] ++(1,2.8) -- +(0,-1.7);
  \path (P) ++(1,2) node [label=left:$p$] {};
  \path (P) ++(2.6,0) node [label=below:pushing forward $p$ along $\sigma$] {};

\end{tikzpicture}
\end{center}

\begin{proof} $\hbox{ }$
\begin{enumerate}
\item From lemma \ref{TwCpullback} we know that $\sigma^*$ is a functor
  $(\Tw(\C_p)_v^\Sub\downarrow B)\rightarrow (\Tw(\C_p)_v\downarrow A)$, so it
  remains to show that $\sigma^*$ maps sub-maps to sub-maps.  This follows from
  the explicit computation of $\sigma^*$ in the proof of lemma \ref{TwCpullback}
  and axiom (P) which gives us that pulling back along a horizontal morphism in
  $\C$ preserves pullbacks.  The fact that $\sigma^*$ preserves covers is true
  by axiom (C).
\item We will show that $\sigma^*$ has a right adjoint by showing that the comma
  category $(\sigma^*\downarrow A')$ has a terminal object.  We will write $A =
  \SCob{a}{i}$, $A' = \SCob{a'}{i'}$, etc.  In addition, for any vertical
  morphism $f:\bd Y=\SCob{y}{w}  & \rSub &  \SCob{z}{x}\ed$ we will use the
  notation $Y_x$ for the object $\{y_w\}_{w\in f^{-1}(x)}$.

  Suppose that we have a sub-map $q:\bd B' & \rSub & B\ed$ such that the
  pullback $\sigma^*q$ factors through $A'$.  Then for all $i\in I$ we have
  horizontal morphisms $\Sigma_i^{-1}:\bd b_{\sigma(i)} & \rTo & a_i\ed$, so by the
  definition of pullback we have sub-maps $\bd B'_{\sigma(i)} & \rSub &
  (\Sigma_i^{-1})^*A_i'\ed$ and thus we must have a sub-map 
  \begin{diagram} B'_{\sigma(i)} & \rSub & \prod_{i\in \sigma^{-1}(j)}
    (\Sigma_i^{-1})^* A'_i.
  \end{diagram}
  (As vertically we are in a preorder, products and pullbacks are the same when
  they exist; we are omitting the object that we take the pullback over for
  conciseness of notation.)
  Thus the object 
  $$X = \coprod_{j\in J}\left(\prod_{i\in \sigma^{-1}(j)} (\Sigma_i^{-1})^*
    A_i'\right)$$ is clearly terminal in $(\sigma^*\downarrow A')$, and if $\bd A'&\rCover&A\ed$ was a cover, then $\bd X&\rCover & B\ed$ will also be one.  (Note that
  if $\sigma^{-1}(j) = \initial$ then the product becomes $\{b_j\}$, as all of
  these products are in the category of objects with sub-map to $B$.)  If $\sigma$ had an injective set map then $\sigma^{-1}(j)$ has size
  either $0$ or $1$ we must have $\sigma^*\sigma_* = 1$.  If $\sigma$ has a surjective set map then by definition $A'_i = \Sigma_i^*B_j$ for $i\in \sigma^{-1}(j)$, and so in fact $X \cong B'$ and $\sigma_*\sigma^*\cong 1$.
\end{enumerate}
\end{proof}

We wrap up this section by defining the category of polytope complexes.

\begin{definition}
Let $\C$,$\D$ be two polytope complexes.  A \textsl{polytope functor} $F:\C\rightarrow \D$ is a double functor satisfying the two additional conditions
\begin{itemize}
\item[(FC)]  the vertical component $F_v:\C_v\rightarrow \D_v$ is continuous and preserves pullbacks and the initial object, and 
\item[(FP)] $F$ preserves pullbacks.  More concretely, given a vertical morphism $P:\bd B' & \rDotsto & B\ed$ and a horizontal morphism $\Sigma:\bd A & \rTo & B\ed$ in $\C$, we have $F(\Sigma^*P) = F(\Sigma)^*F(P)$.
\end{itemize}
We will denote the category of polytope complexes and polytope functors by $\mathbf{PolyCpx}$.
\end{definition}

\section{Waldhausen Category Structure}

Now that we have developed some machinery for looking at formal sums of polygons we can start constructing the group completion of our category $\Tw(\C_p)$.  In order to get this we will use Waldhausen's machinery for categories with cofibrations and weak equivalences (see \cite{waldhausen83} for more details); we denote the category of Waldhausen categories and exact functors by $\mathbf{WaldCat}$.  Our cofibrations will be inclusions of polygons which lose no information.  Our weak equivalences will be the horizontal isomorphisms, together with vertical covering sub-maps (which will quotient out by both of the relations we are interested in).  Since we now want to be able to mix sub-maps and shuffles we define our Waldhausen category by applying a sort of Q-construction to the double category $\Tw(\C_p)$.

\begin{definition} The category $\SC(\C)$ is defined to have
$\ob\SC(\C) = \ob(\Tw(\C_p))$.  The morphisms of $\SC(\C)$ are equivalence
classes of diagrams in $\Tw(\C_p)$
\begin{diagram} A & \lSub^p & A' & \rTo^\sigma & B\end{diagram}
where two diagrams are considered equivalent if there is a vertical
isomorphism $\iota:\bd A'_1&\rSub & A'_2\ed\in \Tw(\C_p)_v$ which makes the
following diagram commute:
\begin{diagram}[small] & &A_1'&  & \\ A &\ldSub(2,1)^{p_1}
&\dSub_{\iota}& \rdTo(2,1)^{\sigma_1} & B\\ & \luSub(2,1)_{p_2}&A_2'&
\ruTo(2,1)^{\sigma_2} & 
\end{diagram}
We will generally refer to a morphism as a \textsl{pure sub-map}
(resp. \textsl{pure shuffle}) if in some representing diagram the shuffle
(resp. sub-map) component is the identity.

We say that a morphism $\bd A & \lSub^p & A' & \rTo^\sigma & B\ed$
is a
\begin{description}
\item[cofibration] if $p$ is a covering sub-map and $\sigma$ has an injective set
map, and a
\item[weak equivalence] if $p$ is a covering sub-map and $\sigma$ has a bijective
set map.
\end{description}
The composition of two morphisms $f:A\rightarrow B$ and $g:B \rightarrow C$ represented by
$$\bd A & \lSub^p & A' & \rTo^\sigma & B \ed \qquad \hbox{and}\qquad \bd B & \lSub^q & B' & \rTo^\tau & C\ed$$
is defined to be the morphism represented by the sub-map $p\circ \sigma^*q$ and the shuffle $\tau \circ \tilde\sigma$.
\end{definition}

Our goal is to prove the following theorem, which states that this structure gives us exactly the scissors congruence groups we were looking for.

\begin{theorem} \lbl{thm:Wald} 
  $\SC$ is a functor $\mathbf{PolyCpx}\rightarrow \mathbf{WaldCat}$.  Every
  Waldhausen category in the image of $\SC$ satisfies the Saturation and
  Extension axioms, and has a canonical splitting for every cofibration
  sequence.   In addition, for any polytope complex $\C$, 
  $K_0\SC(\C)$ is the free abelian group generated by the polytopes of $\C$
  modulo the two relations
  $$[a] = \sum_{i\in I} [a_i] \qquad \textrm{for covering sub-maps }\bd\SCob{a}{i}& \rSub &\{a\}\ed$$
  and
  $$[a] = [b]\qquad \textrm{for horizontal morphisms } \bd a&\rTo& b\ed\in \C.$$
\end{theorem}

We will defer most of the proof of this theorem until the last section of the paper, as it is largely technical and not very illuminating.  Assuming the first part of the theorem, however, we will perform the computation of $K_0$ here.

\begin{proof}
In a small Waldhausen category $\E$ where every cofibration sequence splits, $K_0$ is the free abelian group generated by the objects of $\E$ under the relations that $[A\amalg B] = [A]+[B]$ for all $A,B\in \E$, and $[A] = [B]$ if there is a weak equivalence $\bd A & \rWeakEquiv & B\ed$.

In $\SC(\C)$ an object $\SCob{a}{i}$ is isomorphic to $\coprod_{i\in I} \{a_i\}$, so $K_0\SC(\C)$ is in fact generated by all polytopes of $\C$.  Now consider any weak equivalence $f:\bd A&\rWeakEquiv & B\ed\in\SC(\C)$.  We can write this weak equivalence as a pure covering sub-map followed by a pure shuffle with bijective set map (which will be an isomorphism of $\SC(\C)$).  Any isomorphism of $\SC(\C)$ is a coproduct of isomorphisms between singleton objects; any isomorphism between singletons is simply a horizontal morphism of $\C$.  Any pure covering sub-map is a coproduct of covering sub-maps of singletons.  Thus the weak equivalences generate exactly the relations given in the statement of the theorem, and we are done.
\end{proof}

\section{Examples}

\subsection{The Sphere Spectrum}

Consider the double category $\S$ with two objects, $\initial$
and $*$.  We have one vertical morphism $\bd\initial & \rSub & *\ed$
and no other non-identity morphisms.  There are no nontrivial covers.  Then $\S$
is clearly a polytope complex.

$\Tw(\S)$ will be the category of pointed finite sets, where the cofibrations
are the injective maps and the weak equivalences are the isomorphisms.  By
direct computation and the Barratt-Priddy-Quillen theorem (\cite{barrattpriddy72}) we
see that $K(\SC(\S))$ is equivalent to the sphere spectrum (which
justifies our notation).

\subsection{$G$-analogs of Spheres}

Consider the double category $\S_G$ with two objects, $\initial$ and $*$.  We have one vertical morphism $\bd\initial & \rSub & *\ed$.  In addition, $\Aut_h(*) = G$.  This is clearly a polytope complex.

The objects of $\SC(\S_G)$ are the finite sets.  As all weak equivalences are isomorphisms (and thus all cofibration sequences split exactly) we can compute the K-theory of $\SC(\S_G)$ by considering $\Omega B B(\iso\SC(\S_G))$.  The automorphism group of a set $\{1,2,\ldots,n\}$ is the group $G \wr \Sigma_n$, whose underlying set is $\Sigma_n \times G^n$, and where
$$(\sigma, (g_1,\ldots,g_n)) \cdot (\sigma', (g_1',\ldots, g_n')) = (\sigma \sigma', (g_{\sigma'(1)}g'_1, g_{\sigma'(2)}g'_2, \ldots, g_{\sigma'(n)}g'_n)).$$
Thus the K-theory spectrum of this category will be $\coprod_{n\geq 0} B(G\wr \Sigma_n)$ on the $0$-level, $B(\coprod_{n\geq 0} B(G\wr \Sigma_n))$ on the $1$-level, and an $\Omega$-spectrum above this.

\subsection{Classical Scissors Congruence}

Let $X=E^n,S^n$ or $H^n$ ($n$-dimensional hyperbolic space), and let $\C$ be the
poset of polytopes in $X$.  (A polytope in $X$ is a finite union of
$n$-simplices of $X$; an $n$-simplex of $X$ is the convex hull of $n+1$ points
(contained in a single hemisphere, if $X=S^n$).)  The group $G$ of isometries of
$X$ acts on $\C$; we define a horizontal morphism $P\rightarrow Q$ to be an
element $g\in G$ such that $g\cdot P = Q$.  We say that $\{\bd P_\alpha &\rSub &
P\ed\}_{\alpha\in A}$ is a covering family if $\bigcup_{\alpha\in A} P_\alpha =
P$.  Then $\C$ is a polytope complex.

Then by theorem \ref{thm:Wald} $K_0(\SC(\C)) = \mathcal{P}(X,G)$, the classical
scissors congruence group of $X$.  Thus it makes sense to call $K(\SC(\C))$ the
\textsl{scissors congruence spectrum} of $X$.

\subsection{Sums of Polytope Complexes}

Suppose that we have a family of polytope complexes $\{\C_\alpha\}_{\alpha\in
  A}$.  Then we can define the ``wedge'' of this family by defining $\C$ to be
the double category where $\ob\C = \{\initial\}\cup \bigcup_{\alpha\in A} \ob
\C_{\alpha p}$ (where $\initial$ will be the initial object), and all morphisms
are just those from the $\C_\alpha$.  We define $\Aut_{\C_h}(\initial) =
\bigoplus_{\alpha \in A} \Aut_{h\C_\alpha}(\initial)$.  Then $\C$ will be a
polytope complex which represents the "union" of the polytope complexes in
$\{\C_\alpha\}_{\alpha\in A}$.  $\SC(\C) = \bigoplus_{\alpha\in A} \SC(\C_\alpha)$, where the summation means that all but finitely many of the objects of each tuple will be equal to the zero object.  Then $K(\SC(\C)) = \bigvee_{\alpha\in A} K(\SC(\C_\alpha))$.

%

\subsection{Numerical Scissors Congruence}

Suppose that $K$ is a number field.  Let $\C_K$ be the polytope complex whose objects are the ideals of $\mathcal{O}_K$.  We have a vertical morphism $\bd I & \rSub & J\ed$ if $I|J$, and no nontrivial horizontal morphisms.  (Note that $\mathcal{O}_K$ is the initial object.)  We say that a finite family $\{\bd J_\alpha & \rSub & I \ed\}_{\alpha \in A}$ is a covering family if $I|\prod_{\alpha\in A} J_\alpha$, and an infinite family is a covering family if it contains a finite covering family.  In this case it is easy to compute that $K(\C_K)$ is a bouquet of spheres, one for every prime power ideal of $\mathcal{O}_K$.

Now suppose that $K/\Q$ is Galois with Galois group $G_{K/\Q}$.  We define $\C_{K/\Q}$ to be the polytope complex whose objects and vertical morphisms are the same as those of $\C_K$, but whose horizontal morphisms are induced by the action of $G_{K/\Q}$ on $\C_K$.  Then the K-theory $K(\C_{K/\Q})$ will be 
$$\bigvee_{p^\alpha\in \Z} K(\S_{D_\mathfrak{p}}),$$
where for each $p$, $\mathfrak{p}$ is a prime ideal of $K$ lying above $p$ and $D_\mathfrak{p}$ is the decomposition group of $\mathfrak{p}$.  (For more information about factorizations of prime ideals, see for example \cite{stein04}.)  Thus this spectrum encodes all of the inertial information for each prime ideal of $\Q$.

From the inclusion $\Q\rightarrow K/\Q$ we get an induced polytope functor $\C_\Q \rightarrow \C_{K/\Q}$, and therefore an induced map $f:K(\C_\Q)\rightarrow K(\C_{K/\Q})$.  To compute this map, it suffices to consider what this map does on every sphere in the bouquet $K(\C_\Q)$.  Consider the sphere indexed by a prime power $p^\alpha$.  If we factor $p = \mathfrak{p}_1^e \cdots \mathfrak{p}_g^e$ then this sphere maps to $g$ times the map $K(\S)\rightarrow K(\S_{D_\mathfrak{p}})$ (induced by the obvious inclusion $\S\rightarrow \S_{D_\mathfrak{p}}$), with target the copy of this indexed by $p^{e\alpha}$.  Thus $f$ encodes all of the splitting and ramification data of the extension $K/\Q$.  In particular, we see that the map $K(\C_\Q)\rightarrow K(\C_{K/\Q})$ contains all of the factorization information contained in $K/\Q$.

Note that we can generalize the definition of $\C_{K/\Q}$ to any Galois extension $L/K$; with this definition $\C_K = \C_{K/K}$.

\section{Proof of theorem \ref{thm:Wald}}

This section consists mostly of a lot of technical calculations which check that $\SC(\C)$ satisfies all of the properties that theorem \ref{thm:Wald} claims.  In order to spare the reader all of the messy details we isolate these results in their own section.

\subsection{Some technicalities about sub-maps and shuffles}

In this section all diagrams are in $\Tw(\C_p)$, and all vertical morphisms are sub-maps.

The following lemmas formalize the idea that we can often commute shuffles and sub-maps past one another.  In particular, it is obvious that if we have a sub-map whose codomain is a horizontal pushout, then we can pull this sub-map back to the components of the pushout.  However, it turns out that we can do this in the other direction as well: given suitably consistent sub-maps to the components of a pushout, we obtain a sub-map between pushouts.

\begin{lemma} \lbl{lem:coverpushout}
Given a diagram
\begin{diagram}[small]
C' & \lTo^{\tau'} & A' & \rTo^{\sigma'} & B'\\
\dSub^r && \dSub^p && \dSub_q \\
C & \lTo^\tau & A & \rTo^\sigma & B
\end{diagram}
where the right-hand square is a pullback square and $\sigma$ has an injective
set map, there is an induced sub-map $\bd C'\cup_{A'}B' & \rSub & C\cup_AB\ed$.
If $p,q,r$ are all covering sub-maps then this map will be, as well.
\end{lemma}

\begin{proof} 
  Consider the right-hand square.  Write $A = \SCob{a}{i}$, $B= \SCob{b}{j}$
  (and analogously for $A'$, $B'$).  Write $J = I \cup J_0$ for $J_0 = J
  \backslash \im I$; and $J' = I' \cup J_0'$, analogously.  We claim that $q$
  can be written as $q_I\cup q_0$, where $q_I:\bd \{b'_{j'}\}_{j'\in I'} &\rSub
  & \{b_j\}_{j\in I}\ed$ and $q_0:\bd\{b'_{j'}\}_{j'\in J_0'} & \rSub&
  \{b_j\}_{j\in J_0}\ed$.  Indeed, $q_I$ is well-defined because the diagram
  commutes, and $q_0$ is well-defined because the right-hand square is a
  pullback.  But then we can write the given diagram as the coproduct of two
  diagrams
\begin{diagram}[small]
C' & \lTo^{\tau'} & A' & \rTo^{\sigma'} & \{b'_{j'}\}_{j'\in I'}  &&&\initial & \lTo &  \initial & \rTo & \{b'_{j'}\}_{j'\in J_0'} \\
\dSub^r && \dSub^p && \dSub_{q_I} &&& \dSub && \dSub && \dSub_{q_0} \\
C & \lTo^\tau & A & \rTo^\sigma & \{b_j\}_{j\in I} &&& \initial & \lTo & \initial & \rTo & \{b_j\}_{j\in J_0}
\end{diagram}
As the statement obviously holds in the right-hand diagram, it suffices to
consider the case of the left-hand diagram, where $\sigma$ is bijective.  In
this case, $\sigma$ and $\sigma'$ are both isomorphisms, and so in fact the
morphism we are interested in is in fact $r$, and so the lemma is trivial.
\end{proof}

Suppose that we are given a diagram
\begin{diagram}[small]
   A' & \rTo^{\sigma'} & B' & \lSub^{f'} & C'\\
   \dSub_p && \dSub_q && \dSub_r\\
   A & \rTo^{\sigma} & B &\lSub^f & C
 \end{diagram}
Then from the definition of $\sigma^*$ and $(\sigma')^*$ it is easy to see that
we get an induced sub-map $\bd(\sigma')^*C' & \rSub & \sigma^*C\ed$, which will
be a covering sub-map if $p,q,r$ are.  The analogous statement for $\sigma_*$ is
more difficult to prove, but is also true.

\begin{lemma} \lbl{lem:covermaps}
Given a diagram
\begin{diagram}[small]
    C' & \rSub^{f'} & A' & \rTo^{\sigma'} & B' \\
    \dSub^r && \dSub_p && \dSub_q \\
     C & \rSub^{f} & A & \rTo^{\sigma} & B
  \end{diagram}
there is an induced sub-map $\bd \sigma'_* C' & \rSub & \sigma_* C\ed$, which is
a covering sub-map if $p,q,r$ are all covering sub-maps.
\end{lemma}

\begin{proof} 
We can assume that $\sigma$'s set map is surjective, since otherwise we can write the right-hand square as the coproduct of two squares
\begin{diagram}[small]
A' & \rTo^{\sigma'} & B'_0 &\qquad& \initial & \rTo & B'_1 \\
\dSub && \dSub && \dSub && \dSub \\
A & \rTo^\sigma & B_0 && \initial & \rTo & B_1
\end{diagram}
In the right-hand case the map we are interested in is just $\bd B'_1 & \rSub &
B_1\ed$, so the result clearly holds.  So we focus on the original question when $\sigma$ has a surjective set-map.
As $(\Tw(\C_p)_v^\Sub\downarrow B)$ is a preorder and both $\sigma'_*C'$ and $\sigma_*C$ sit over $B$ it suffices to show that this morphism exists in $(\Tw(\C_p)_v^\Sub\downarrow B)$.  We claim that it suffices to show that $\sigma'_*C' = \sigma_*C'$, as if this is the case then
$$\Hom_{(\Tw(\C_p)_v^\Sub\downarrow B)}(\sigma'_*C', \sigma_*C) = \Hom_{(\Tw(\C_p)_v^\Sub\downarrow B)}(\sigma_*C', \sigma_*C) \supseteq \Hom_{(\Tw(\C_p)_v^\Sub\downarrow A)}(C', C) \neq\initial$$
so we will be done.

Write $A = \SCob{a}{i}$, $A' = \SCob{a'}{i'}$, $B = \SCob{b}{j}$, etc., and define $C'_i =
\{c'_{k'}\}_{pf'(k')=i}$ for $i\in I$ and $C'_{i'} = \{c'_{k'}\}_{f'(k') = i'}$
for $i'\in I'$.  Then we know that 
$$\sigma_*C' = \coprod_{j\in J} \prod_{i\in \sigma^{-1}(j)} \Sigma_i^*
C'_i \qquad\hbox{and}\qquad \sigma'_*C' = \coprod_{j'\in J'} \prod_{i'\in
  {\sigma'}^{-1}(j')} {\Sigma'}_{i'}^* C'_{i'}.$$ 
(Note that all of these products exist because $\C$ is vertically closed under
pullbacks, and in a preorder a pullback is the same as a product.)  Now
$$\prod_{i\in \sigma^{-1}(j)} \Sigma_i^*C_i' = \prod_{i\in \sigma^{-1}(j)}
\coprod_{i'\in p^{-1}(i)} \Sigma_{p(i')}^* C'_{i'}$$
since we know that if $p(i'_1) = p(i'_2)$ then $\sigma'(i'_1) \neq
\sigma'(i'_2)$.  But because the diagram commutes,
$$\prod_{i\in \sigma^{-1}(j)}
\coprod_{i'\in p^{-1}(i)} \Sigma_{p(i')}^* C'_{i'} = \coprod_{j'\in q^{-1}(j)}
\prod_{i'\in {\sigma'}^{-1}(j')} \Sigma_{p(i')}^* C'_{i'} = \coprod_{j'\in q^{-1}(j)}
\prod_{i'\in {\sigma'}^{-1}(j')} {\Sigma'_{i'}}^* C'_{i'}.$$
Thus 
$$\sigma_*C' = \coprod_{j\in J} \prod_{i\in \sigma^{-1}(j)} \Sigma_i^*
  C'_i = \coprod_{j\in J} \coprod_{j'\in q^{-1}(j)} \prod_{i'\in
    {\sigma'}^{-1}(j')} {\Sigma'}_{i'}^* C'_{i'} = \sigma'_*C'$$
as desired.
\end{proof}

Lastly we prove a couple of lemmas which show that covering sub-maps do not lose
any information.  The first of these shows that if two shuffles are related by
covering sub-maps then they contain the same information; the second shows that
pulling back a covering sub-map along a shuffle cannot lose information.

\begin{lemma} \lbl{lem:mapbtwncovers}
Suppose that we have the following diagram:
\bd
A' & \rTo^\tau & B' \\
\dCover^p && \dCover_q \\
A & \rTo^\sigma & B
\ed
Then this diagram is a pullback square, and $\tau$ is an isomorphism if and only if $\sigma$ is.
\end{lemma}

\begin{proof}
Pulling back $q$ along $\sigma$ gives us a commutative square
\bd A' & \rTo^\tau & B' \\ \dCover^r && \dSub_\cong \\ \sigma^*B' &
\rTo^{\sigma'} & B'
\ed
so it suffices to show that in any such diagram $r$ is an isomorphism.  Suppose it were not.  Then there would exist $a'\in A'$ and an $a\in \sigma^*B'$ such that we have a non-invertible vertical morphism $\bd a' & \rSub & a\ed$, and horizontal morphisms $F_a:\bd a & \rTo & b\ed$ (for some $b\in B'$) such that the pullback of the vertical identity morphism on $b$ is the non-invertible morphism $\bd a'&\rSub & a\ed$.  Contradiction.  Thus $r$ must be an isomorphism, and we are done with the first part.

If $\sigma$ is an isomorphism then any pullback of it will also be an
isomorphism, so if $\sigma$ is an isomorphism then so is $\tau$.  Now suppose
that $\tau$ is an isomorphism.  As any shuffle with bijective set map is an
isomorphism it suffices to show that $\sigma$ has a bijective set map.  However,
as this is a pullback square on the underlying set maps we can just consider it
there. $q$ has a surjective set map (as it is a cover) and the pullback of
$\sigma$ along $q$ is a bijection, so $\sigma$ must also be a bijection, and we
are done.

\end{proof}

\subsection{Checking the axioms}

We now verify that our definition of $\SC(\C)$ works and then check the axioms for it to be a Waldhausen category.  First we check that all of our definitions are well-defined.

\begin{lemma}
  $\SC(\C)$ is a category, and the cofibrations and weak equivalences form subcategories of $\SC(\C)$.
\end{lemma}

\begin{proof}
We need to check that composition is well-defined.  Suppose that we are given morphisms $f:A\rightarrow B$ and $g:B\rightarrow C$ in $\SC(\C)$, and suppose that we are given two different diagrams representing each morphism.  Then we have the following diagram, where the top and bottom squares are pullbacks:
\begin{diagram}[small]
&&&& \sigma_1^*B'_1 \\
&& A'_1 &\ldSub(2,1)^{q'_1} && \rdTo(2,1)^{\sigma'_1} & B'_1\\
A & \ldSub(2,1)^{p_1} && \rdTo(2,1)^{\sigma_1} &B & \ldSub(2,1)^{q_1} && \rdTo(2,1)^{\tau_1} & C \\
& \luSub(2,1)_{p_2} & A'_2 & \ruTo(2,1)_{\sigma_2} & & \luSub(2,1)_{q_2} & B'_2 & \ruTo(2,1)_{\tau_2} \\
&& & \luSub(2,1)_{q_2'} & \sigma_2^*B'_2 & \ruTo(2,1)_{\sigma'_2}
\end{diagram}
As each vertical section represents the same map we have reindexings $\iota_A:A_1\rightarrow A_2$ and $\iota_B:B_1\rightarrow B_2$; we need to show that we therefore have a reindexing $\sigma_1^*B_1'\rightarrow \sigma_2^*B_2'$.  It is easy to see that pulling back a reindexing along either a sub-map or a shuffle produces another reindexing.  Thus if we pull back $\iota_A$ along $q_2'$ to get a morphism $\iota_A'$, and then pull back $\iota_B$ along $\iota_A'\sigma_2'$ to get $\iota_B'$ we get a diagram
\begin{diagram}[small]
&&&& \sigma_1^*B_1' \\
A & \lSub & A_1' & \rTo \ldSub(2,1) & B & \lSub \rdTo(2,1) & B_1' & \rTo & C \\
&&& \luSub(2,1) & X & \ruTo(2,1)
\end{diagram}
where $\iota_2'\iota_1':X\rightarrow\sigma_2^*B_2'$.  However, as both the upper
and lower squares are pullbacks they are unique up to unique vertical
isomorphism, so we obtain a vertical isomorphism $\bd X & \rSub & \sigma_1^*
B_1'\ed$, and we are done.

It remains to show that weak equivalences and cofibrations are preserved by composition.  Consider a composition of morphisms determined by the following diagram:
\begin{diagram}[small]
&&&& \sigma^*B' \\
&& A' &\ldSub(2,1)^{q'} && \rdTo(2,1)^{\sigma'} & B'\\
A & \ldSub(2,1)^{p} && \rdTo(2,1)^{\sigma} &B & \ldSub(2,1)^{q} && \rdTo(2,1)^{\tau} & C 
\end{diagram}
If $q$ is a cover then so is $q'$, so if both $p$ and $q$ are covers then $q'p$ is also a cover.  From the formula in lemma \ref{lem:pullbacks} it is easy to see that if a shuffle has an injective (resp. bijective) set map then so will its pullback, so if both $\sigma$  and $\tau$ are injective (resp. bijective) then $\tau\sigma'$ will be as well.  Thus cofibrations and weak equivalences form subcategories, as desired.
\end{proof}

\begin{lemma}
 Any isomorphism is both a cofibration and a weak equivalence.
\end{lemma}

\begin{proof}
Suppose that $f:A\rightarrow B$ is an isomorphism with inverse $g:B\rightarrow A$.  In $\Tw(\C_p)$ $f$ and $g$ are represented by diagrams
\begin{diagram}
A & \lSub^p & A' & \rTo^\sigma & B &\qquad & B & \lSub^q & B' & \rTo^\tau & A 
\end{diagram}
respectively.  As $gf=1_A$ we must have $p\circ\sigma^*(q)$ be invertible, so in particular $p$ must be an isomorphism; thus we can pick a diagram representing $f$ such that $p = 1_A$ (which is in particular a covering sub-map).  Applying the analogous argument to $g$ we can see that we can pick a diagram representing $g$ such that $q = 1_B$.  In that case, it is easy to see that we must have $\tau = \sigma^{-1}$ in $\Tw(\C_p)_h$, so $\sigma$ and $\tau$ must be invertible.  From this we see that any isomorphism is both a cofibration and a weak equivalence, as desired.
\end{proof}

Now we move on to proving some of the slightly more complicated axioms defining a Waldhausen category.  We check that pushouts along cofibrations exist, and that they preserve cofibrations.  In fact, in $\SC(\C)$ pushouts not only preserve cofibrations; they also preserve weak equivalences.

\begin{lemma} \lbl{lem:pushouts}
 Given any diagram
\begin{diagram}
C & \lTo & A & \rCofib^f & B
\end{diagram}
the pushout $C\cup_A B$ of this diagram exists, and the morphism $C\rightarrow C\cup_A B$ is a cofibration.  If $f$ were also a weak equivalence, then this map would also be a weak equivalence.
\end{lemma}

\begin{proof}
The diagram above is represented by the following diagram in $\Tw(\C_p)$:
\begin{diagram}[small]
&& A'' && && A' \\
C & \ldTo(2,1)^\tau && \rdSub(2,1)^q &A& \ldCover(2,1)^p && \rdTo(2,1)^\sigma & B
\end{diagram}
Complete the middle part of this diagram to a pullback square
\begin{diagram}[small]
\tilde A & \rSub^{q'} & A' \\
\dCover^{p'} && \dCover^p\\
A'' & \rSub^q & A
\end{diagram}
Now we define $s:\bd C'& \rCover & C\ed$ to be $\tau_*(q')$.  As $\tau^*\tau_*(p')$ must factor through $p'$, $q\circ \tau^*\tau_*(p')$ must factor through $p$ (as $\Tw(\C_p)_v^\Sub$ is a preorder), so we can write $q\circ \tau^*\tau_*(p') = p\circ r'$; we define $r:\bd B' & \rSub & B\ed$ to be $\sigma_*(r')$.  (Note that if $\tau$ has an injective set map then $r = \sigma_*(q')$.)  Now we have the following diagram in $\Tw(\C_p)$:
\begin{diagram}[small]
C' & \lTo^{\tau'} &A ''' & \rEqual & A''' & \rTo^{\sigma'} & B' \\
\dCover^s && \dCover && \dSub& & \dSub^{r} \\
C & \lTo^\tau & A'' && A' & \rTo^\sigma & B
\end{diagram}
where both squares are pullback squares.  The top row of this diagram consists
only of maps in $\Tw(\C_h)$.  As $\C_h$ is a groupoid it in particular has all
pushouts, and so by lemma \ref{lem:twpushouts} the pushout $C' \cup_{A'''} B'$
exists in $\Tw(\C_p)$; we claim that gives us the pushout of the original
diagram.  Note that the set-map of the shuffle $\bd C' & \rTo & C'\cup_{A'''}
B'\ed$ will be injective (and bijective, if $\sigma$'s was) because it is the
pushout of $\sigma'$, so the pushout of a cofibration (resp. weak equivalence)
is another cofibration (resp. weak equivalence).

To check that this is indeed the pushout, suppose that we are given any commutative square 
\begin{diagram}[small]
A& \rCofib^f & B\\
\dTo && \dTo \\
C & \rTo & D
\end{diagram}
The diagonal morphism $A\rightarrow D$ is represented by a diagram $\bd A & \lSub^t & \hat A & \rTo^\rho & D\ed$ in $\Tw(\C_p)$.  Considering the composition around the top, we see that $\hat A$ factors through $p$, and considering the composition around the bottom it must factor through $q$.  In addition, as $t$ comes from the bottom composition we know that $\sigma^*\sigma_*t = t$ and thus $t$ must factor through $A'''$.  We can now apply lemma \ref{lem:coverpushout} to see that we indeed get a unique factorization through our pushout, as desired.
\end{proof}

We have now shown that $\SC(\C)$ is a category with cofibrations which is equipped with a subcategory of weak equivalences, and we move on to proving that all cofibration sequences split canonically.  Given a cofibration $f:A\hookrightarrow B$ we say that the \textsl{cofiber} of $f$ is the pushout of $f$ along the morphism $A\rightarrow *$.  We will denote such a map by
\begin{diagram} B & \rFib & B/A\end{diagram}

\begin{corollary} \lbl{cor:fibersection}
Any cofiber map has a canonical section; this section is a cofibration.
\end{corollary}

\begin{proof}
  Suppose that we are given a cofiber map $\bd B & \rFib & B/A\ed$.  Suppose
  that this map is represented by the diagram
  \begin{diagram}
    B & \lCover^p & B' & \rTo^\sigma & B/A
  \end{diagram}
  From the computation in the proof of lemma \ref{lem:pushouts} it is easy to
  see that if we write $B = \SCob{b}{j}$ then $B' = \{b_j\}_{j\in J'}$ where
  $J'\subseteq J$ and $\sigma$ has a bijective set-map.  Thus we can define a
  pure shuffle $\sigma^{-1}:\bd B/A & \rCofib & B\ed$ which will be our section.
  If we change the diagram representing the fiber map by a reindexing then
  $\sigma^{-1}$ changes exactly by this reindexing, so this construction is
  well-defined.
\end{proof}

\begin{remark} 
This construction is canonical in the twisted arrow category whose objects are
cofibrations of $\SC(\C)$.  It is not canonical in the ordinary arrow category.
\end{remark}

Now it only remains to show that the weak equivalences of $\SC(\C)$ satisfy all
of the axioms we desire of a Waldhausen category.

\begin{lemma} \lbl{lem:extension}
$\SC(\C)$ satisfies the extension axiom.  In other words, given any diagram
\begin{diagram}[small]
A & \rCofib & B & \rFib & B/A \\
\dWeakEquiv && \dTo^f & & \dWeakEquiv \\
A' & \rCofib & B' & \rFib & B'/A'
\end{diagram}
$f$ is also a weak equivalence.
\end{lemma}

\begin{proof}
It is easy to see that the two sections given by corollary \ref{cor:fibersection} are going to be consistent in the sense that they will split the above diagram into two diagrams
\begin{diagram}[small]
A & \rAcycCofib & \tilde A &&\qquad&& B_0 & \rTo^\cong & B/A \\
\dWeakEquiv && \dTo_{f_A} &&&& \dTo^{f_B} && \dWeakEquiv \\
A' &\rAcycCofib & \tilde A' &&&& B_0' & \rTo^\cong & B'/A' 
\end{diagram}
where $f$ will equal $f_A \sqcup f_B$ up to isomorphism.  Thus it suffices to
show that both $f_A$ and $f_B$ are weak equivalences.  That $f_B$ is a weak
equivalence is obvious from the diagrams.  Up to isomorphism, the left-hand
square is represented by the following diagram in $\Tw(\C_p)$:
\begin{diagram}[small]
  A & \lCover & \tilde A \\
  \uCover && & \luSub(2,1)^p & \hat A \\
  A' & \lCover & \tilde A' & \ldTo(2,1)_\sigma
\end{diagram}
In order to show that $f_A$ is a weak equivalence we need to show that $p$ is a
covering sub-map and that $\sigma$ is a horizontal isomorphism.  If we pull the
covering sub-map $\bd \tilde A' & \rCover & A \ed$ back along $\sigma$ we will
get the sub-map $\bd \hat A & \rSub & A\ed$; thus this composition is a covering
sub-map, and in particular so is $p$.  Now applying lemma
\ref{lem:mapbtwncovers} we see that $\sigma$ must therefore be an
isomorphism.
\end{proof}

\begin{lemma} \lbl{lem:gluing}
$\SC(\C)$ satisfies the gluing axiom. In other words, given any diagram
\begin{diagram}[small]
C & \lTo & A & \rCofib & B \\
\dWeakEquiv & & \dWeakEquiv && \dWeakEquiv\\
C' & \lTo & A' & \rCofib & B'
\end{diagram}
the induced morphism $C\cup_A B \rightarrow C'\cup_{A'} B'$ is also a weak equivalence.
\end{lemma}

\begin{proof}
  It is a simple calculation to see that it suffices to consider
  diagrams represented in $\Tw(\C_p)$ by 
  \begin{diagram}[small]
    C & \lTo^\sigma & A_C & \rSub & A & \lCover & B \\
    \uCover^p &&&& \uCover && \uCover \\
    C' &\lTo^{\sigma'} & A'_C & \rSub & A' & \lCover & B'
  \end{diagram}
  By pulling back $p$ along $\sigma$ we get a cover $\bd A'_C &
  \rCover & A_C\ed$.  Then we have a cover $\bd
  A'_C\times_{A'} B' & \rCover & A_C\times_A B\ed$, and thus a diagram
  \begin{diagram}[small]
    C & \lTo^\sigma & A_C & \lSub & A_C\times_A B \\
    \uCover && \uCover && \uCover \\
    C' & \lTo^{\sigma'} & A_C' & \lSub & A'_C\times_{A'} B'
  \end{diagram}
  to which we can apply lemma \ref{lem:covermaps}, thus giving
  us a cover $\bd \sigma'_*(A'_C\times_{A'} B') & \rCover &
  \sigma_*(A_C\times_A B)\ed$.  However, a simple computation using
  the definition of the pushout shows us that this is exactly the morphism
  between the pushouts of the top and bottom row.  As it is
  represented in $\Tw(\C_p)$ by a covering sub-map, it must be a weak
  equivalence, as desired.
\end{proof}

\begin{lemma} \lbl{lem:saturation}
$\SC(\C)$ satisfies the saturation axiom. In other words, given any two composable morphisms $f,g$, if two of $f,g,gf$ are weak equivalences then so is the third.
\end{lemma}

\begin{proof}
Let $f:A\rightarrow B$ and $g:B\rightarrow C$ be morphisms in $\SC(\C)$.  As weak equivalences form a subcategory of $\SC(\C)$ we already know that if $f$ and $g$ are both weak equivalences then so is $gf$.  So it suffices for us to focus on the other two cases.  In the following discussion we will be considering the following diagram
\begin{diagram}[small]
&&&& \sigma^*B' & \\
& & A' &  \ldSub(2,1)^{q'} && \rdTo(2,1)^{\sigma'} & B' &  \\
A &\ldSub(2,1)^p &&\rdTo(2,1)^\sigma& B &\ldSub(2,1)^q && \rdTo(2,1)^\tau& C
\end{diagram}
where the middle square is a pullback.

First, suppose that $gf$ and $g$ are weak equivalences.  Then we know that both $\tau$ and $\tau\sigma'$ are isomorphisms, which means that $\sigma'$ must be as well.  In addition, $pq'$ is a covering sub-map which means (by axiom (B)) that $p$ must be as well.  As $q$ is a cover, by lemma \ref{lem:mapbtwncovers} $\sigma$ therefore must also be an isomorphism.  As $p$ is a covering sub-map and $\sigma$ is an isomorphism, $f$ is also a weak equivalence.

Now suppose that $gf$ and $f$ are weak equivalences.  Then we know that $pq'$ is a covering sub-map, which means that $q'$ must be a covering sub-map as well.  We know that $q' = \sigma^*q$, and as $\sigma$ is an isomorphism we must have $q = \sigma^{-1})^* q'$.  As covering sub-maps are preserved by pullback $q$ must also be a covering sub-map.  In addition, by lemma \ref{lem:mapbtwncovers} as $\sigma$ is an isomorphism so is $\sigma'$, and thus (as $\tau\sigma'$ is an isomorphism) $\tau$ must be as well.  Thus we see that $q$ is a covering sub-map and $\tau$ an isomorphism, so $g$ is a weak equivalence as desired.
\end{proof}

Lastly we prove that our construction is in fact functorial.

\begin{proposition}
A polytope functor $F:\C\rightarrow \D$ induces an exact functor of Waldhausen categories $\SC(F):\SC(\C)\rightarrow \SC(\D)$.
\end{proposition}

\begin{proof}
As $\Tw$ is a functorial construction, a double functor $F:\C\rightarrow \D$ induces a double functor $\Tw(F):\Tw(\C)\rightarrow \Tw(\D)$.  Then we define the functor $\SC(F):\SC(\C)\rightarrow \SC(\D)$ to be the one induced by $\Tw(F)$.  This is clearly well-defined on objects.  As $F_v$ preserves pullbacks and the initial object, $\Tw(F)_v$ takes sub-maps to sub-maps, and thus $\SC(F)$ is well-defined on morphisms. Composition in $\SC(\C)$ is defined by pulling back vertical  morphisms along horizontal morphisms, which commutes with $F$ as $F$ is a polytope functor, so $\SC(F)$ is a well-defined functor.  If $f$ is a (vertical or horizontal) morphism in $\Tw(\C)$ then $F(f)$ and $f$ have the same set-map, so $\Tw(F)$ preserves injectivity and bijectivity of set-maps, so to see that $\SC(F)$ preserves cofibrations and weak equivalences it suffices to check that $\Tw(F)$ preserves covering sub-maps, which it must as $F_v$ is continuous.  So in order to have exactness it suffices to show that $\SC(F)$ preserves pushouts along cofibrations.

Suppose that we are given the diagram
\begin{diagram}
C & \lTo & A & \rCofib & B
\end{diagram}
in $\SC(\C)$, which is represented by the diagram
\begin{diagram}[small]
&& A' && && A'' && && \\
C &\ldTo(2,1)^\tau & &\rdDotsto(2,1)^q& A & \ldCover(2,1)^p& &\rdTo(2,1)^\sigma& B
\end{diagram}
in $\Tw(\C)$.  The pushout of this is computed by computing the pullback of the two sub-maps, pushing forward the result along $\tau$, pulling it back along $\tau$, and then pushing it foward along $\sigma$.  Thus in order to see that $\SC(F)$ preserves pushouts it suffices to show that $\Tw(F)$ commutes with pullbacks of sub-maps, and pulling back or pushing forward a sub-map along a shuffle.  By considering the formulas for pullbacks and pushforwards we see they they consist entirely of pulling back squares in $\C$ and taking vertical pullbacks in $\C$, so we are done.
\end{proof}

\begin{remark}
  The proof of lemma \ref{lem:saturation} is the only place in this proof that the full power of
  axiom (B) was used.  The only other place where we used axiom (B) was to show that covering sub-maps are preserved by pullbacks, and there the only fact we used was that removing the morphism $\bd\initial & \rSub & A\ed$ from a covering family of $A$ still keeps it a covering family.  So if we had a double category $\C$ which satisfied all
  of the axioms except for (B) and satisfied this covering condition, we would still have a Waldhausen category
  $\SC(\C)$, except that it would not satisfy one direction of the Saturation
  axiom.  (If $gf$ and $f$ were weak equivalences then $g$ would still be a weak equivalence.)
  In that case we could still define the K-theory of $\SC(\C)$, but our computations would be more complicated,   We keep axiom (B), however, because it
  holds in all of our motivating examples, and because the Saturation axiom
  makes many of our computations simpler.
\end{remark}

\bibliographystyle{plain}
\bibliography{IZ}

\end{document}